\documentclass[12pt,leqno,fleqn]{amsart}
\usepackage{amsmath,amstext,amsthm,amssymb,amsxtra}
\usepackage{txfonts,pxfonts} 
\usepackage[colorlinks,citecolor=red,pagebackref,hypertexnames=false]{hyperref}
\usepackage[backrefs,author-year]{amsrefs}
\usepackage{pgf}
\usepackage{pgflibraryarrows}
\usepackage{tikz}
\allowdisplaybreaks

\theoremstyle{plain} 
\newtheorem{lemma}[equation]{Lemma}

\newtheorem{theorem}[equation]{Theorem}

 \newtheorem{conjecture}[equation]{Conjecture}
\newtheorem{smallBallConjecture}[equation]{Small Ball Conjecture}
\newtheorem{signedSmallBallConjecture}[equation]{Signed Small Ball Conjecture}

\newtheorem{talagrand}[equation]{Talagrand's Theorem}
\newtheorem{bg}[equation]{The  Simplest Instance of the Beck Gain}
\newtheorem{lp}[equation]{Littlewood-Paley Inequalities}

\theoremstyle{definition}
\newtheorem{definition}[equation]{Definition}

\theoremstyle{remark}

\numberwithin{equation}{section}

\def\norm#1.#2.{\lVert#1\rVert_{#2}}
\def\Norm#1.#2.{\bigl\lVert#1\bigr\rVert_{#2}}
\def\NOrm#1.#2.{\Bigl\lVert#1\Bigr\rVert_{#2}}
\def\NORm#1.#2.{\biggl\lVert#1\biggr\rVert_{#2}}
\def\NORM#1.#2.{\Biggl\lVert#1\Biggr\rVert_{#2}}

\def\ip#1,#2,{\langle #1,#2\rangle}
\def\Ip#1,#2,{\bigl\langle#1,#2\bigr\rangle}
\def\IP#1,#2,{\Bigl\langle#1,#2\Bigr\rangle}

\def\mid{\,:\,}

\def\abs#1{\lvert#1\rvert}
\def\Abs#1{\bigl\lvert#1\bigr\rvert}
\def\ABs#1{\biggl\lvert#1\biggr\rvert}

\def\XXint#1#2#3{{\setbox0=\hbox{$#1{#2#3}{\int}$}
     \vcenter{\hbox{$#2#3$}}\kern-.5\wd0}}

\begin{document}
\title[Signed Small Ball Inequality]
{On the Signed Small Ball Inequality}
\author[D.~Bilyk]{Dmitriy Bilyk}
\address{School of Mathematics \\ Georgia Institute of Technology \\ Atlanta GA 30030}

\email{bilyk@math.gatech.edu}

\author[M.\thinspace T.~Lacey]{Michael T. Lacey}
\address{School of Mathematics \\ Georgia Institute of Technology \\ Atlanta GA 30030}

\email{lacey@math.gatech.edu}

\author[A. Vagharshakyan]{Armen Vagharshakyan}
\address{School of Mathematics \\ Georgia Institute of Technology \\ Atlanta GA 30030}

\email{armenv@math.gatech.edu}

\begin{abstract}
Let $ h_R$ denote an $ L ^{\infty }$ normalized Haar function
adapted to a dyadic rectangle $ R\subset [0,1] ^{d}$. We show that
for all choices of coefficients $ \alpha (R)\in \{\pm1\}$, we have the
following lower bound on the $ L ^{\infty }$ norms of the sums of
such functions, where the sum is over rectangles of a fixed volume.
\begin{equation*}
 n ^{\eta (d)} \lesssim \NOrm \sum _{\abs{ R}= 2 ^{-n}} \alpha (R) h_R (x).
 L ^{\infty} ([0,1] ^{d}) .
\,, \quad \textup{for all }\quad \eta (d) < \frac {d-1} 2 + \frac 1
{8d}\,,
\end{equation*}
where the implied constant is independent of $ n\ge 1$. The
inequality above (without  restriction on the coefficients) arises
in connection to several areas, such as Probabilities,
Approximation, and Discrepancy. With $\eta (d)= (d-1)/2$, 
the inequality above follows from orthogonality, 
 while it is conjectured that the inequality holds  with
$\eta (d) ={d/2}$. This is known and proved in \cite{MR95k:60049} in
the case of $ d=2$, and
 recent papers of the of the authors \cite{bl}, \cite{blv} prove
 that in higher dimensions one can
take $ \eta (d)> (d-1)/2$, without specifying a particular  value of $ \eta
$. The restriction $\alpha_R \in \{\pm1\}$ allows us to
significantly simplify our prior arguments and to find an explicit value of
$\eta(d)$.
\end{abstract}

\maketitle

\section{The Small Ball Conjectures} 

 In one dimension, the class of dyadic intervals are $\mathcal D {} \coloneqq {}\{ [j2^k,(j+1)2^k)\mid j,k\in \mathbb Z\} $.
 Each dyadic interval has a left and right half, which are also dyadic.  Define the
 $L^\infty$-normalized Haar functions
 \begin{equation*}
h_I \coloneqq -\mathbf 1 _{I_{\textup{left}}}+ \mathbf 1
_{I_{\textup{right}}}.
\end{equation*}

 In $d$ dimensions, a \emph{dyadic rectangle} is a product of dyadic intervals, i.e. an
element of
 $\mathcal D^d $.   We define a Haar function associated to $R $  to  be the product of the Haar functions associated
 with each side of $R $, namely
 \begin{equation*}
 h_{R_1\times\cdots\times R_d }(x_1,\ldots,x_d) {} \coloneqq {}\prod _{j=1}^d h _{R_j}(x_j).
 \end{equation*}

 We will consider a local problem and concentrate on rectangles with  fixed volume.
 This is the `hyperbolic'
 assumption, that pervades the subject.
 Our concern is the following Theorem and Conjecture concerning a
 \emph{lower bound} on the $ L ^{\infty }$ norm of sums of hyperbolic Haar functions:

 \begin{smallBallConjecture} \label{small} For dimension $ d\ge 3$ we have the inequality
 \begin{equation}\label{e.Talagrand}
2 ^{-n} \sum _{\abs{ R}= 2 ^{-n} } \abs{ \alpha(R) }
{}\lesssim{} n ^{\frac12(d-2) } \NOrm \sum _{\abs{ R}\ge 2 ^{-n} } \alpha (R) h_R .\infty .
\end{equation}
 \end{smallBallConjecture}

 Average case analysis --- that is passing through $ L ^2 $  --- shows that we always have
\begin{equation*}
2 ^{-n} \sum _{\abs{ R}= 2 ^{-n} } \abs{ \alpha(R) }
{}\lesssim{} n ^{\frac12(d-1) } \NOrm \sum _{\abs{ R}\ge 2 ^{-n} } \alpha (R) h_R .\infty .
\end{equation*}
Namely, the constant on the right is bigger than in the conjecture by a factor
of $ \sqrt n$.
 We refer to this as the `average case estimate,' and refer to improvements over this
 as a `gain over the average case estimate.'

Random choices of coefficients $ \alpha (R)$ show that the Small Ball Conjecture is sharp.
The interest in this conjecture arises from questions in Probability Theory
\cite{MR95k:60049}, Approximation Theory \cite{MR1005898} and the
theory of Irregularities of Distribution \cite{MR903025}.

In dimension $ d=2$, the Conjecture was resolved by \cite{MR95k:60049}.\footnote{This
result should be compared to \cite{MR0319933}, as well as \cite{MR96c:41052}.}

 \begin{talagrand}\label{j.talagrand}
    In dimension $d=2 $, we have
 \begin{equation}  \label{e.talagrand}
2 ^{-n} \sum _{\abs{ R}= 2 ^{-n} } \abs{ \alpha(R) } {}\lesssim{}
\NOrm \sum _{\abs R \ge 2 ^{-n}} \alpha (R) h_R .\infty .
 \end{equation}
 Here, the sum on the right is taken over all rectangles with area \emph{at least } $ 2
 ^{-n}$.
 \end{talagrand}

 In dimensions  $ d\ge 3$, there is partial information in \cite{bl}, \cite{blv}, which builds upon
 the method devised by \cite{MR1032337}.

\begin{theorem}\label{t.bl} In dimension $ d\ge 4$, there is a  $ \zeta(d)>0  $ so that for all
choices of coefficients $ \alpha (R)$ we have
 \begin{equation}  \label{e.bl}
2 ^{-n} \sum _{\abs{ R}= 2 ^{-n} } \abs{ \alpha(R) } {}\lesssim{} n
^{\frac{d-1}{2}- \zeta(d)   } \NOrm \sum _{\abs R \ge 2 ^{-n}}
\alpha (R) h_R .\infty .
 \end{equation}
\end{theorem}

 The Conjecture \ref{small} appears to be quite difficult to resolve in dimensions $ d\ge3$,
 and in this paper we discuss a   more restrictive formulation of the conjecture that still
 appears to be of interest.

\begin{signedSmallBallConjecture}
\label{c.restricted}
We have the inequality   (\ref{e.Talagrand}),
in the case where the coefficients $ \alpha (R)\in \{\pm 1\}$, for $\abs{ R}= 2 ^{-n} $
. Namely, under these assumptions on the
coefficients $ \alpha (R)$ we have the inequality
\begin{equation}\label{e.signed}
 n ^{d/2} \lesssim \NOrm \sum _{\abs{ R}= 2 ^{-n}} \alpha (R) h_R. \infty .
\,.
\end{equation}
\end{signedSmallBallConjecture}

The main result of this note is the next Theorem, in which we give
an explicit gain over the trivial bound in the Signed Small Ball
Conjecture in dimensions $ d\ge 3$.

\begin{theorem}\label{t.d=3} In dimension $ d\ge 3$, for
choices of coefficients $ \alpha (R)\in \{\pm1\}$, we have the inequality
\begin{equation}\label{e.d=3}
 n ^{\eta (d) } \lesssim \NOrm \sum _{\abs{ R}= 2 ^{-n}} \alpha (R) h_R. \infty .
\,, \quad \textup{for all} \quad \eta (d) <  \frac {d-1} 2 + \frac 1
{8d}\,.
\end{equation}
\end{theorem}

The interest in the Theorem above is that the amount of the gain is
explicit, and that the method of proof, using essential ingredients
from \cites{blv,bl,MR1032337}, is much simpler than either of
these prior works.

The principal difficulty in three and higher dimensions is that two
dyadic rectangles of the same volume can share a common side length.
Beck \cite{MR1032337} found a specific estimate in this case, an
estimate that is extended in \cites{blv,bl}. The reader is
encouraged to consult \cite{blv} for a more detailed exposition of
the methods and applications. The main simplification in the current
note lies in the equalities \eqref{e.main}, which allow us to avoid
analyzing longer coincidences. The value of $ \eta $ appears to be
the optimal one we can get out of this line of reasoning, imputing
additional interest to the methods of proof used to improve this
estimate.

\section{Notations and Littlewood-Paley Inequality} 
Let $\vec r\in \mathbb N^d$ be a partition of $n$, thus $\vec r=(r_1
,\dotsc,  r_ d)$, where the  $r_j$ are nonnegative integers and
$\abs{ \vec r} \coloneqq \sum _t r_t=n$, which we refer to as the
\emph{length of the vector $ \vec r$}. Denote all such vectors as $
\mathbb H _n$. (`$ \mathbb H $' for `hyperbolic.') For vector $ \vec
r $ let $ \mathcal R _{\vec r} $ be all dyadic rectangles $ R$ such
that for each coordinate $ k$, $ \lvert  R _k\rvert= 2 ^{-r_k} $.

\begin{definition}\label{d.rfunction}
We call a function $f$ an \emph{$\mathsf r$ function  with parameter $ \vec r$ } if
\begin{equation}
\label{e.rfunction} f=\sum_{R\in \mathcal R _{\vec r}}
\varepsilon_R\, h_R\,,\qquad \varepsilon_R\in \{\pm1\}\,.
\end{equation}
A fact used without further comment is that $ f _{\vec r} ^2 \equiv 1$.
\end{definition}

The $\mathsf r$ functions we are interested in are:
\begin{equation}\label{e.fr}
f _{\vec r} \coloneqq \sum _{R\in \mathcal R _{\vec r}}  \alpha  (R) \, h_R
\end{equation}

\bigskip

We recall those Littlewood-Paley inequalities of most interest to
us. Notice that due to the $L^\infty$ normalization some of the
equalities here will look odd to a reader accustomed to the $L^2$
normalization.
\begin{lp}  In one dimension, we have the inequalities
\begin{equation}\label{e.lp}
\NOrm \sum _{I\subset \mathbb R } a_I h_I (\cdot) .p.
\lesssim
\sqrt p \NORm \biggl[ \sum _{I\subset \mathbb R }  { a_I ^2 }
\mathbf 1 _{I} (\cdot )\biggr] ^{1/2} .p.\,,
\qquad 2<p<\infty \,.
\end{equation}
Moreover, these inequalities continue to hold in the case where
the coefficients $ a_I $ take values in a Hilbert space $ \mathcal H$.
\end{lp}

The growth of the constant is essential for us, in particular the
factor $ \sqrt p$ is, up to a constant, the best possible in this
inequality.  See \cites{MR1439553,MR1018577}.  That these
inequalities hold for Hilbert space valued sums is imperative for
applications to higher dimensional sums of Haar functions. The
relevant inequality is as follows.

\begin{theorem}\label{t.LP} We have the inequalities below
for hyperbolic sums of $ \mathsf r$ functions in   dimension $ d\ge 3$.
\begin{equation}\label{e.LP}
\NOrm \sum _{\lvert  \vec r\rvert =n} f _{\vec r} .p.
\lesssim (p n) ^{ (d-1)/2}\,, \qquad 2<p<\infty \,.
\end{equation}
\end{theorem}





\section{Proof of Theorem~\ref{t.d=3}} 

The proof of the Theorem is by duality, namely we construct a
function $ \Psi $ of $ L ^{1}$ norm about one, which is used to
provide a lower bound on the $ L ^{\infty }$ norm of the sum of Haar
functions.

The function $ \Psi $ will take the form of a Riesz product, but in
order to construct it, we need some definitions first. Fix $
0<\kappa <1$, with the interesting choices of $ \kappa $ being close
to zero. Define relevant parameters by
\begin{gather} \label{e.q}
q= \lfloor a n ^{\varepsilon} \rfloor \,,\qquad \varepsilon = \frac
1 {2d}- \kappa \,, \qquad   b =\tfrac 14\,,
\\
\label{e.rho}
\widetilde \rho=a  q ^{b}  n^{- (d-1)/2}\,, \qquad \rho =  {\sqrt q} n ^{-(d-1)/2}.
\end{gather}
Here $ a $ is a small positive constant, we use the notation  $
b=\tfrac 14$  throughout, so as not to obscure  those aspects of the
argument that that dictate these  choices of parameters.
 $ \widetilde \rho $ is  a `false' $ L^2$
 normalization for the sums we consider, while the larger term $ \rho $ is the
 `true' $ L ^{2}$ normalization.
Our `gain over the average case estimate' in the Small Ball Conjecture is $ q ^{b}
 \simeq n ^{\varepsilon /4}=n ^{1/8d -\kappa/4}=n ^{\eta (d)- (d-1)/2}$.

Divide the integers $ \{1,2,\dotsc,n\}$ into $ q$ disjoint increasing  intervals 
of equal length 
$ I_1,\dotsc, I_q$, and let $ \mathbb A _t \coloneqq \{\vec r\in \mathbb H _n
\mid r_1\in I_t\}$.  Let
\begin{equation}
\label{e.G_t}
F_t \coloneqq   \sum _{\vec r\in \mathbb A _t}  f _{\vec r}\,,
\qquad
H \coloneqq \sum _{\vec r\in \mathbb H _n} f_r=\sum _{t=1} ^{q} F_t\,.
\end{equation}
The Riesz product is a `short product.'
\begin{equation*}
\Psi \coloneqq \prod _{t=1} ^{q} (1+  \widetilde  \rho F_t) \,,
\qquad
\Psi _{\neq j} \coloneqq \prod _{\substack{t=1\\ t\neq j }} ^{q} (1+  \widetilde  \rho F_t) \,,
\quad 1\le j \le q\,.
\end{equation*}
Note the subtle way that the false $ L^2$ normalization enters into
the product. It means that the product is, with high probability,
positive.  And of course, for a positive function $ F$, we have $
\mathbb E F=\norm F.1.$, with expectations being typically easier to
estimate. This heuristic is made precise below. Notice also that
 $\mathbb E \Psi = 1$.

We need a final bit of notation.  Set
\begin{equation}\label{e.Phi}
\Phi _{t} \coloneqq  \sum _{\substack{\vec r\neq \vec s\in \mathbb A _t\\ r _{1}= s _{1} } }
f _{\vec r} \cdot f _{\vec s}\,.
\end{equation}
Note that in this sum, there are $ 2d-3$ free parameters among the vectors $ \vec r$ and
 $ \vec s$.  That is, the pair of vectors $ (\vec r,\vec s)$ are completely specified by
 there values in $ 2d-3$ coordinates.

Our main Lemma is below.  Note that in (\ref{e.3/2}), the assertion
is that the $ 2d-3$ parameters in the definition of $ \Phi _t$
behave, with respect to $ L ^{p}$ norms, as if they are independent.

\begin{lemma}\label{l.Tech} We have these estimates.
\begin{align}\label{e.1}
\norm \Psi .1. & \lesssim 1 \,,
\\  \label{e.2}
\norm \Psi .2.\,,\  \max _{1\le j\le n} \,\norm \Psi _{\neq j}.2.& \lesssim
\operatorname e ^{a ' q ^{2b}}\,,
\\ \label{e.3/2}
\norm \Phi _t .p. &\lesssim  p ^{d-1/2}  n ^{d-3/2} q ^{-1/2}\,,
\qquad 2<p<\infty \,.
\end{align}
In (\ref{e.2}), the value of $ a'$ is a decreasing function of $ 0<a<1$.
\end{lemma}

The proof of this Lemma is taken up  next section. Assuming the
Lemma, we proceed as follows. An important simplification in the
Signed Small Ball Inequality comes from the equalities
\begin{equation} \label{e.main}  
\begin{split}
\ip F_j, \Psi, &= \ip \sum_{\vec r \in \mathbb A_j} f_r, \Psi ,
\\
&= \sum_{\vec r \in
\mathbb A_j} \ip f_r , (1+ \widetilde \rho F_j) \Psi _{\neq j} ,
\\
&= \widetilde \rho \sum_{\vec r \in \mathbb A_j} \ip  f_r^2 ,
\Psi_{\neq j}, + \widetilde \rho \ip \Phi_j  ,  \Psi _{\neq j} ,
\\
&= \widetilde \rho \, \sharp \mathbb A_j + \widetilde \rho \ip \Phi
_j, \Psi _{\neq j} , \, .
\end{split}
\end{equation}
We have used the fact that there has to be a coincidence in the
first coordinate in order for the product of $\mathsf r$ functions
to have non-zero integral. The first term in the third line is the
`diagonal' term, while the second term arises from different vectors
which coincide in the first coordinate. Therefore, we can estimate
\begin{align*}
\norm H. \infty . & \gtrsim
\ip H, \Psi ,
\\
&=\sum _{j=1} ^{q} \ip F_j, \Psi ,
\\
&= \widetilde \rho \, \sharp  \mathbb H_n + \sum _{j=1} ^{q}
\widetilde \rho \ip \Phi _{j}, \Psi _{\neq j},
\end{align*}

\smallskip

It is clear that
\begin{equation} \label{e.PE}
 \widetilde \rho \, \sharp  \mathbb H_n \simeq  a^{5/4} n^{\frac{d-1}{2}+\frac{\varepsilon}{4}}=a
^{5/4}  n ^{\eta (d) }\,,
\end{equation}
which is our principal estimate. The other term we treat as an error term.
Using H\"older's inequality, and (\ref{e.1}) and (\ref{e.2}) we see that
\begin{equation*}
\norm \Psi _{\neq j}. (q ^{2b})'.  \le
  \norm \Psi  _{\neq j}. 1. ^{(q^{2b}-2)/q ^{2b}}
  \norm \Psi _{\neq j}.2. ^{ 2q ^{-2b}}
  \lesssim 1\,.
\end{equation*}
Therefore, we can estimate as below, where we use the estimate above and
(\ref{e.3/2}).
\begin{align*}
\ABs{\sum _{j=1} ^{q} \widetilde \rho \ip \Phi _{j}, \Psi _{\neq j},
} & \lesssim \sum _{j=1} ^{q} \widetilde \rho  \norm \Phi _j. q
^{2b}. \norm \Psi _{\neq j}.  (q ^{2b})'.
 \\&\lesssim  q  \cdot  \frac {a q ^{b}} {n ^{(d-1)/2}}\cdot q ^{2b(d-1/2)}  n^{d-3/2} \cdot q ^{-1/2}  \\&
 \simeq a q ^{2b d + 1/2} n ^{(d-2)/2}
 \ll   n ^{\eta (d)}\,.
\end{align*}
 This term will
be smaller than the term in (\ref{e.PE}).  The proof of our main
result is complete, modulo the proof of Lemma~\ref{l.Tech}.

\section{The Analysis of the Coincidence and Corollaries of the Beck Gain} 

Following the language of J.~Beck \cite{MR1032337}, a \emph{coincidence} occurs if we have two
vectors $ \vec r\neq \vec s$ with e.\thinspace g.~$ r_1=s_1$,
precisely the condition that we imposed in the definition of $ \Phi _t$, (\ref{e.Phi}).
He observed that sums over
products of $ \mathsf r$ functions in which there are coincidences obey favorable
$ L^2$ estimates.  We refer to (extensions of) this observation as the \emph{Beck Gain.}

\begin{bg}\label{l.SimpleCoincie}
We have the estimates below, valid for an absolute implied constant
that is only a function of dimension $ d \ge 3$.
\begin{equation}\label{e.BG}
\sup _{1\le j\le n}  \norm \Phi _{j}.p. \lesssim   p ^{d-1/2} \cdot
n ^{d-3/2} q ^{-1/2}\,, \qquad  1\le p < \infty \,.
\end{equation}
\end{bg}

This Lemma, in dimension $ d=3$ appears in \cite{bl}.  The proof in
higher dimensions, which was given in \cite{blv}, is inductive. We
omit the proof here as it is rather lengthy and refer the reader to
\cite{blv} for the details. The estimate in the aforementioned papers
does not strictly speaking contain Lemma \ref{l.SimpleCoincie}, as
it does not include $q^{-1/2}$. However, this can be easily fixed in
the proof due to the fact that the value of the first coordinate can
be chosen in $n/q$ ways rather than $n$. We also emphasize that the
estimate above may admit an improvement, in that the power of $ p$
is perhaps too large by a single power.

\begin{conjecture}\label{j.bg}
We have the estimates below, valid for an absolute implied constant
that is only a function of dimension $ d \ge 3$.
\begin{equation}\label{e.BG}
\sup _{1\le j\le n}  \norm \Phi _{j}.p. \lesssim   (pn) ^{d-3/2} q
^{-1/2}\,, \qquad  1\le p < \infty \,.
\end{equation}
\end{conjecture}

With this conjecture we could prove our main theorem for all $  \eta
(d)< \frac {d-1} 2 + \frac 1 {8d-8}$.

The Beck Gain Lemma \ref{l.SimpleCoincie} has several important
consequences. Theorem~\ref{t.LP} implies an exponential estimate for
sums of $ \mathsf r$ functions. However, with the Beck Gain at hand, we
can derive a subgaussian estimate for such sums, for moderate
deviations.

\begin{theorem}\label{t.better} Using the notation of (\ref{e.rho}) and (\ref{e.G_t}),
we have this estimate, valid for all $ 1\le t\le q$.
\begin{equation}\label{e.better}
\norm  \rho F_t  .p.  \lesssim  \sqrt p \,  \,, \qquad 1\le p \le c
\left(\frac{n}{q} \right)^{\frac {1 } {2d-1}} \,.
\end{equation}
As a consequence, we have the distributional estimate
\begin{equation}\label{e.partialX2}
\mathbb P ( \lvert  \rho F _t\rvert  >x ) \lesssim \operatorname
{exp} (- c x ^2 )\,, \qquad x< c  \left(\frac{n}{q} \right) ^{
\frac {1 } {4d-2}}\,.
\end{equation}
Here $ 0<c<1$ is an absolute constant.
\end{theorem}


\begin{proof}

Recall that
\begin{equation*}
F_t =  \sum _{\vec r\in \mathbb A _t}  f _{\vec r}\,.
\end{equation*}
where  $ \mathbb A _t \coloneqq \{\vec r\in \mathbb H _n
\mid r_1\in I_t\}$, and $ I_t$  in an interval of integers of length $ n/q$,
so that $ \sharp \mathbb A _t \simeq   n ^2 /q \simeq \rho ^{-2}$.

Apply the Littlewood-Paley inequality  in the first coordinate.
This results in the estimate
\begin{align*}
\norm \rho F_t .p.
& \lesssim \sqrt p
\NOrm \Bigl[\sum _{s\in I_j} \Abs{ \rho \sum _{\vec r\,:\, r_1=s} f _{\vec r}} ^2
\Bigr] ^{1/2} .p.
\\
& \lesssim  \sqrt p \norm   1+   \rho ^{2}\Phi _{t}   .p/2. ^{1/2}
\\
& \lesssim  \sqrt p \Bigl\{ 1 + \norm \rho ^{2}\Phi _{t}   .p/2.
^{1/2}\Bigr\} \,.
\end{align*}
Here, it is important to use the constants in the Littlewood-Paley
inequalities that give the correct order of growth of $ \sqrt p$. Of
course the terms $ \Phi _{t}$ are controlled by the estimate in
(\ref{e.BG}). In particular, we have
\begin{equation}\label{e.FFFjjj}
\norm \rho ^{2}\Phi _{t}.p. \lesssim  \frac{ q } {n ^{d-1}}  p
^{d-1/2}  n ^{d-3/2}  q ^{-1/2} \lesssim p ^{d-1/2} n ^{-1/2}
q^{1/2}\,.
 \end{equation}
Hence (\ref{e.better}) follows.

\smallskip

The second distributional inequality is a well known consequence of the norm
inequality.  Namely, one has the inequality below, valid for all $ x$:
\begin{equation*}
\mathbb P (\rho F_t >x ) \le C ^{p} p ^{p/2} x ^{-p}\,, \qquad 1\le
p \le c \left(\frac{n}{q} \right) ^{\frac {1 } {2d-1}} \,.
\end{equation*}
If $ x$ is as in (\ref{e.partialX2}), we can take $ p \simeq x ^{2}$ to prove the
claimed exponential squared bound.
 \end{proof}

\begin{proof}[Proof of (\ref{e.1}).]

Observe that
\begin{equation*}
\mathbb P (\Psi <0) \lesssim q \operatorname {exp} (c a ^{-2} q ^{1-2b})\,.
\end{equation*}
Indeed, using (\ref{e.partialX2}), we have
\begin{align*}
\mathbb P (\Psi <0)&\le \sum _{t=1} ^{q} \mathbb P ( \widetilde \rho \,  F_t < -1)
\\
& = \sum _{t=1} ^{q}  \mathbb P (\rho F_t < - a ^{-1} q ^{1/2-b})
\\
& \lesssim q \operatorname {exp} (c a ^{-2} q ^{1-2b})\,.
\end{align*}
Note that to be able to use (\ref{e.partialX2}) we need to have $a
^{-1} q ^{1/2-b} \le c \left(\frac{n}{q} \right) ^{\frac {1 }
{2d-1}}$, which leads to $\varepsilon<\frac{1}{2d}$.  Then, assuming
(\ref{e.2}), we have
\begin{align*}
\norm \Psi .1. & = \mathbb E \Psi -2 \mathbb E \Psi \mathbf 1 _{\Psi <0}
\\
& \le 1+ 2\mathbb P (\Psi <0) ^{1/2} \norm \Psi .2.
\\
& \lesssim 1+ \operatorname {exp} ( - a ^{-2}q ^{1-2b}/2+ a q ^{ 2b})\,.
\end{align*}
For sufficiently small $ 0<a<1$, the proof is finished.  Note that this
last step forces $ b= 1/4$ on us.

\end{proof}

\begin{proof}[Proof of (\ref{e.2}).]
The supremum over $ j$ will be an immediate consequence of the proof below, and
so we don't address it specifically.

Let us give the initial, essential observation.
We expand
\begin{equation*}
\mathbb E \prod _{j=1} ^{q} (1+ \widetilde \rho F_j) ^2
=
\mathbb E \prod _{j=1} ^{q} (1+ 2\widetilde \rho F_j+ (\widetilde \rho F_j) ^2  )\,.
\end{equation*}
Hold the $ x_2$ and $ x_3$ coordinates fixed, and let $ \mathcal F$ be the sigma
field generated by $ F_1 ,\dotsc, F_{q-1}$.  We have
\begin{equation} \label{e.;p}
\begin{split}
\mathbb E \bigl(1+ 2\widetilde \rho F_q+ (\widetilde \rho F_q) ^2
\,\big|\, \mathcal F\bigr) &=1+\mathbb E \bigl((\widetilde \rho F_q) ^2
\,\big|\, \mathcal F\bigr)
\\
&=1+   a ^2  q ^{2b-1}+ \widetilde \rho ^2  \Phi _{q}\,,
 \end{split}
\end{equation}
where $ \Phi _{q}$ is defined in (\ref{e.Phi}).
 Then, we see that
\begin{align} \nonumber
\mathbb E \ \prod _{v=1} ^{q} (1+ 2\widetilde \rho F_t+ (\widetilde
\rho F_t) ^2  ) &= \mathbb E \Bigl\{  \prod _{v=1} ^{q-1} (1+
2\widetilde \rho F_t+ (\widetilde \rho F_t) ^2  )\, \times \mathbb E
\bigl(1+ 2\widetilde \rho F_t+ (\widetilde \rho F_t) ^2 \,\big|\,
\mathcal F \bigr)\Bigr\}
\\
&\le  \label{e.;;}
(1+a ^2 q ^{2b-1})
\mathbb E \prod _{v=1} ^{q-1} (1+ 2\widetilde \rho F_t+ (\widetilde \rho F_t) ^2  )
\\  \label{e.;;;;}
& \qquad + \mathbb E \abs{ \widetilde\rho ^{2} \Phi _{q} } \cdot
\prod _{v=1} ^{q-1} (1+ 2\widetilde \rho F_t+ (\widetilde \rho F_t)
^2  )
\end{align}
This is the main observation: one should induct on (\ref{e.;;}),
while treating the term in (\ref{e.;;;;}) as an error, as the `Beck
Gain' estimate (\ref{e.BG}) applies to it.

Let us set up notation to implement this line of approach.  Set
\begin{equation*}
N (V;r) \coloneqq \NOrm \prod _{v=1} ^{V} (1+ \widetilde \rho F_v)
.r.  \,, \qquad   V=1 ,\dotsc, q\,.
\end{equation*}
We will obtain a very crude estimate for these numbers for $ r=4$.
Fortunately, this is relatively easy for us to obtain. Namely,
$ q$ is small enough that we can use the inequalities (\ref{e.better}) to see that
\begin{align*}
N (V; 4)&\le \prod _{v=1} ^{V} \norm 1+ \widetilde \rho F_v .4V.
\\
& \le ( 1 +  C q ^{b}  ) ^{V}
\\
& \le (Cq) ^{   q}\,.
\end{align*}
For a large choice of $ \tau >1$, which is a function of the choice of
$ \kappa >0$ in (\ref{e.q}), we
have the estimate below from H\"older's inequality
\begin{equation}\label{e.killq^q}
N (V;2(1-1/\tau q) ^{-1} )\le N (V;2) ^{1-2/\tau q} \cdot N (V; 4)
^{2/\tau q}\,.
\end{equation}

We see that (\ref{e.;;}), (\ref{e.;;;;}) and (\ref{e.killq^q})
give us the inequality
\begin{equation}\label{e.==}
\begin{split}
N (V+1; 2) ^2  & \le (1+a ^2 q ^{2b-1})    N (V; 2) ^2 + C  \cdot  N
(V; 2 (1-1/\tau q) ^{-1} ) ^2   \cdot \norm  \widetilde \rho ^{2}
\Phi _{V} .\tau  q.
\\
& \le (1+a ^2 q ^{2b-1})    N (V; 2)^2 + C  N (V;2) ^{2-4/\tau q}
\cdot N (V; 4) ^{4/\tau q} \norm  \widetilde \rho ^{2} \Phi _{V} .
\tau q.
\\
& \le (1+a ^2 q ^{2b-1})   N (V; 2) ^2 + C_ \tau   q ^{d-1/2+4/\tau
} n ^{-1/2} N (V;2) ^{2-2/\tau q} \,.
\end{split}
\end{equation}
In the last line we have used the the inequality (\ref{e.BG})
and the constant $ C _{\tau }$ is only a function of $ \tau >1$, which is fixed.

Of course we only apply this as long as $ N (V; 2)\ge 1$.  Assuming this is true
for all $ V\ge 1$, we see that
\begin{equation*}
N (V+1; 2) ^2 \le (1+a ^2 q ^{2b-1} + C_ \tau   q ^{d-1/2+4/\tau } n
^{-1/2} ) N (V; 2) ^2 \,.
\end{equation*}
And so, by induction,
\begin{align*}
N (q;2)&\lesssim  (1+a ^2  q ^{2b-1}+  C_ \tau   q ^{d-1/2+4/\tau }
n ^{-1/2} ) ^{q/2} \lesssim \operatorname e ^{ 2a q ^{2b}}\,.
\end{align*}
Here, the last inequality will be true for large $ n$, provided
$ \tau $ is much bigger than $ 1/\kappa $.  Indeed, we need
\begin{align*}
a ^2  q ^{2b-1}&\ge   C_ \tau   q ^{d-1/2+4/\tau } n ^{-1/2}
\end{align*}
Or equivalently,
\begin{equation*}
c n ^{1/2}\ge q ^{d+4/\tau }\,.
\end{equation*}
Comparing to the definition of $ q$ in (\ref{e.q}), we see that the proof is finished.

\end{proof}


\begin{bibsection}
 \begin{biblist}

\bib{MR1032337}{article}{
    author={Beck, J{\'o}zsef},
     title={A two-dimensional van Aardenne-Ehrenfest theorem in
            irregularities of distribution},
   journal={Compositio Math.},
    volume={72},
      date={1989},
    number={3},
     pages={269\ndash 339},
      issn={0010-437X},
    review={MR1032337 (91f:11054)},
}



\bib{MR903025}{book}{
    author={Beck, J{\'o}zsef},
    author={Chen, William W. L.},
     title={Irregularities of distribution},
    series={Cambridge Tracts in Mathematics},
    volume={89},
 publisher={Cambridge University Press},
     place={Cambridge},
      date={1987},
     pages={xiv+294},
      isbn={0-521-30792-9},
    review={MR903025 (88m:11061)},
}


\bib{bl}{article}{
 author={Bilyk, Dmitry},
 author={Lacey, Michael T.},
 title={On the Small  Ball  Inequality in Three Dimensions},
  eprint={arXiv:math.CA/0609815},
  date={2006},
 }


\bib{blv}{article}{
 author={Bilyk, Dmitry},
 author={Lacey, Michael T.},
  author={Vagharshakyan, Armen},
 title={On the Small  Ball  Inequality in All Dimensions},
   date={2007},
 }

\bib{MR1439553}{article}{
    author={Fefferman, R.},
    author={Pipher, J.},
     title={Multiparameter operators and sharp weighted inequalities},
   journal={Amer. J. Math.},
    volume={119},
      date={1997},
    number={2},
     pages={337\ndash 369},
      issn={0002-9327},
    review={MR1439553 (98b:42027)},
}



\bib{MR0319933}{article}{
   author={Schmidt, Wolfgang M.},
   title={Irregularities of distribution. VII},
   journal={Acta Arith.},
   volume={21},
   date={1972},
   pages={45--50},
   issn={0065-1036},
   review={\MR{0319933 (47 \#8474)}},
}



\bib{MR95k:60049}{article}{
    author={Talagrand, Michel},
     title={The small ball problem for the Brownian sheet},
   journal={Ann. Probab.},
    volume={22},
      date={1994},
    number={3},
     pages={1331\ndash 1354},
      issn={0091-1798},
    review={MR 95k:60049},
}

\bib{MR96c:41052}{article}{
    author={Temlyakov, V. N.},
     title={An inequality for trigonometric polynomials and its application
            for estimating the entropy numbers},
   journal={J. Complexity},
    volume={11},
      date={1995},
    number={2},
     pages={293\ndash 307},
      issn={0885-064X},
    review={MR 96c:41052},
}



\bib{MR1005898}{article}{
   author={Temlyakov, V. N.},
   title={Approximation of functions with bounded mixed derivative},
   journal={Proc. Steklov Inst. Math.},
   date={1989},
   number={1(178)},
   pages={vi+121},
   issn={0081-5438},
   review={\MR{1005898 (90e:00007)}},
}

\bib{MR1018577}{article}{
   author={Wang, Gang},
   title={Sharp square-function inequalities for conditionally symmetric
   martingales},
   journal={Trans. Amer. Math. Soc.},
   volume={328},
   date={1991},
   number={1},
   pages={393--419},
   issn={0002-9947},
   review={\MR{1018577 (92c:60067)}},
}




  \end{biblist}
 \end{bibsection}
 \end{document}